\documentclass[12pt]{amsart}

\usepackage{amsmath}
\usepackage{amssymb}
\usepackage{booktabs}
\usepackage{stmaryrd}
\usepackage{url}
\usepackage{longtable}
\usepackage[figuresright]{rotating}

\input xy
\xyoption{all}

\newcommand{\RR}{\mathbb{R}}
\newcommand{\NN}{\mathbb{N}}
\newcommand{\QQ}{\mathbb{Q}}

\newcommand{\IFF}{\Leftrightarrow}

\newcommand{\continuum}{\mathfrak{c}}

\newcommand{\mc}{\mathcal}
\newcommand{\mf}{\mathfrak}
\newcommand{\bez}{\backslash}

\newtheorem{mydef}{Definition}[section]
\newtheorem{myth}{Theorem}[section]
\newtheorem{myfact}{Proposition}[section]
\newtheorem{mylemma}{Lemma}[section]
\newtheorem{mycoro}{Corollary}[section]
\newtheorem{myremark}{Remark}[section]
\newtheorem{problem}{Problem}

\title[Luzin and Sierpi\'nski sets, ...]{Luzin and Sierpi\'nski sets, some nonmeasurable subsets of the plane}
\author{Marcin Michalski} 
\email[Marcin Michalski]{marcin.michalski@matpwr.info}
\author{Szymon \. Zeberski}
\email[Szymon \. Zeberski]{szymon.zeberski@pwr.edu.pl}
\address{Institute of Mathematics and Computer Science, Wroc\l aw University of Technology, Wybrze\. ze Wyspia\'nskiego 27, 50-370 Wroc\l aw, Poland}
\date{}

\begin{document}
\begin{abstract}
In this paper we shall introduce some nonmeasurable and completely nonmeasurable subsets of the plane with various additional properties, e.g. being Hamel basis, intersecting each line in a super Luzin/Sierpi{\'n}ski set. Also some additive properties of Luzin and Sierpi{\'n}ski sets and their generalization, $\mc{I}$-Luzin sets, on the line are investigated. 
\end{abstract}
\maketitle
\footnotetext{\noindent The work of Sz. {\.Z}eberski has been partially financed by NCN means granted by decision DEC-2011/01/B/ST1/01439.}

\section{Introduction and preliminaries}

Although Luzin and Sierpi{\'n}ski sets were introduced in 1913 and 1924 respectively (see \cite{Lu}, \cite{Mahl}, \cite{Sier}) they are still a fruitful topic to study. They are interesting from many points of view, e.g. as small sets in some sense (see \cite{Filip}, \cite{Kha}) or beacuse they provide existence of other types of sets with certain properties (see \cite{Mi2}). Subsets of the plane in the view of Luzin and Sierpi{\'n}ski sets were also considered recently in \cite{Bien}.

Additive properties of reals are of some intrest too. Let us recall a classical example \cite{Erd} and some recent ones - \cite{Ky} and especially \cite{5P}.

Most of the set-theoretic notation is based on \cite{Je}. A useful reference concerning measure and category is \cite{Ox}.

We say that a family $\mc{I}$ is an ideal of sets, if it is closed under unions and taking subsets. We call $\mc{I}$ an $\sigma$-ideal if it is an ideal and it is closed under countable unions.

Let $\mc{I}$ be a $\sigma$-ideal of subsets of $\RR$ ($\RR^2$) and $\mc{B}$ a family of Borel sets. We say that $\mc{I}$:
\begin{itemize}
\item is translation invariant, if for each $x\in\RR$ and $I\in\mc{I}$ we have $x+I\in\mc{I}$,
\item is scale invariant, if for each $x\in\RR$ and $I\in\mc{I}$ we have $xI\in\mc{I}$,
\item has Borel base if $(\forall I\in\mc{I})(\exists B\in\mc{B}\cap\mc{I})(I\subset B)$ (see \cite{5P}),
\item has Steinhause property if $Int(A-B)\neq\emptyset$ for each $A, B\in\mc{B}\bez\mc{I}$ (see also \cite{Balc}).
\end{itemize}
Classical examples of $\sigma$-ideals on $\RR$ ($\RR^2$), meager sets $\mc{M}$ and null sets $\mc{N}$, have all of the above properties. Throughout this paper we assume that $\sigma$-ideal $\mc{I}$ is translation and scale invariant, has  Borel base and  Steinhaus property and is nontrivial i.e. $\RR (\RR^2 \text{ resp.}) \notin \mc{I}$.
\\
For every $A, B\subset\RR$ we define Minkowski sum of $A$ and $B$ as:
$$
A+B=\{a+b:\ a\in A, b\in B\}.
$$
For each $n\in\NN$ we shall denote:
$$
\underbrace{A+\cdots+A}_{n}=\bigoplus^{n}A.
$$
Considering the algebraic structure of $\RR$ and $\RR^{2}$ we treat them as a linear spaces over the rattionals $\QQ$. 

\begin{mydef}
Let $A\subset\RR$. We say that $A$ is:
\begin{itemize}
\item $\mc{I}$-residual if $A$ is a complement of some set $I\in\mc{I}$,
\item $\mc{I}$-nonmeasurable if $A$ doesn't belong to the $\sigma(\mc{B}\cup\mc{I})$,
\item  completely $\mc{I}$-nonmeasurable if $A$ intersects every $\mc{I}$-positive Borel set (that is Borel set that is not a member of $\mc{I}$) and doesn't contain any of them.
\end{itemize}
\end{mydef}

Let us also recall that we say that $A$ is:
\begin{itemize}
\item a Luzin set if $|L|=\continuum$ and every intersection of $L$ and a meager set is countable,
\item a strong Luzin set if A is a Luzin set and every intersection of $A$ and a $\mc{M}$-positive Borel set is uncountable,
\item a Sierpi\'nski set if $|S|=\continuum$ and every intersection of $S$ and a null set is countable,
\item a strong Sierpi\'nski set if A is a Sierpi\'nski set and every intersection of $A$ and a $\mc{N}$-positive Borel set is uncountable,
\item a Bernstein set if for each perfect set $P$ we have $B\cap P\neq \emptyset$ and $B^{c}\cap P\neq \emptyset$.
\end{itemize}
Let us note that a notion of $\mc{I}$-nonmeasurability and complete $\mc{I}$-nonmeasurability agrees with \cite{R}, \cite{RZ} and \cite{Z}.

Under the assumption of CH one may construct Luzin and Sierpi\'nski sets. On the other hand Martin's Axiom and failure of CH imply that every set of cardinality lesser than $\continuum$ is both meager and null, so Luzin and Sierpi\'nski sets don't exist (see \cite{MA}). Furthermore we have:
\begin{myth}\textbf{(Rothberger, 1938)}
CH $\IFF$ Luzin and Sierpi{\'n}ski sets exist.
\end{myth}
The proof can be found e.g. in \cite{Mi}.
 
Due to Marczewski partition of $\RR$ into a comeager null set and a meager set that has a full measure (see \cite{Ox}, Corollary 1.7) it's quite easy to check that every Luzin set is a null set and every Sierpi{\'n}ski set is meager. Futhermore we have:

\begin{myfact}
	Every Luzin set does not have  the Baire property and every Sierpi\'nski set is Lebesgue nonmeasurable.
\end{myfact}

There is a nice characterization of strong Luzin sets. Since every Borel set is of the form $U\Delta M$, where $U$ is open and $M$ is meager, a Luzin set is a super Luzin set $\IFF$ it has uncountable intersection with every open interval that has rational endpoints.

\begin{myfact}
Every strong Luzin set is completely $\mc{M}$-nonmeasurable and every strong Sierpi{\'n}ski set is completely $\mc{N}$-nonmeasurable.
\end{myfact}
We shall consider some generalization of Luzin and Sierpi\'nski sets. 

\begin{mydef}
We say that $L$ is an $\mc{I}$-Luzin set if $|L|=\continuum$ and for every $I\in\mc{I}$ a set $L\cap I$ is countable. Analogously to previous definitions, we call $L$ a strong $\mc{I}$-Luzin set if $L$ is an $\mc{I}$-Luzin and its intersection with every Borel $\mc{I}$-positive set is uncountable.
\end{mydef}

Such notions were considered e.g. in \cite{RaZe}, \cite{Buk}.

We shall provide a way to achieve strong $\mc{I}$-Luzin set with minimal assumptions, but first we need the following proposition:
\begin{myfact}
\label{oi-rezydualnym}
Let $B$ be a Borel $\mc{I}$-positive set and let $D$ be a countable dense set. Then $B+D$ is an $\mc{I}$-residual set.
\end{myfact}
\begin{proof}
Suppose that a set $B+D$ is not $\mc{I}$-residual. Then A=$(B+D)^c$ is Borel and doesn't belong to $\mc{I}$. By the Steinhaus property there exists an interval $I\subset(B-A)$ and:
$$
\RR=D+I\subset D+(B-A)\subset (D+B)-A,
$$
what brings a contradiction, because $0\notin(D+B)-A$.

\end{proof}

\begin{mylemma}
\label{superiluzinq}
Let $L$ be a $\mc{I}$-Luzin set. Then $L+\QQ$ is a strong $\mc{I}$-Luzin set.
\end{mylemma}

\begin{proof}
$L+\QQ$ is a $\mc{I}$-Luzin set as a countable union of $\mc{I}$-Luzin sets. Let $B$ be a Borel $\mc{I}$-positive set. By Proposition \ref{oi-rezydualnym}  $B+\QQ$ is an $\mc{I}$-residual set and we may assume w.l.g. that $L\subset (B+\QQ)$. Then for each $l\in L$ exist $b\in B$ and $q\in \QQ$ that $l=b+q$ and eventually $b=l-q$. Let us denote an equivalence relation $\sim$ over $L$:
$$
(\forall l, l'\in L)(l\sim l'\IFF l-l'\in\QQ)
$$
If for some $l, l'\in L$, $b\in B$ and $q, q'\in\QQ$ we have $b=l-q=l'-q'$ then $l-l'=q-q'$ and therefore $l\sim l'$. There are $\continuum$ many equivalence classes and therefore $|(L+\QQ)\cap B|=\continuum$.

\end{proof}

We shall end the preliminaries with the following theorem:
\begin{myth}[CH]
	\label{partycjailuzinych}
	 There exists a partition of $\RR$ into $\mf{c}$ many strong $\mc{I}$-Luzin sets.
\end{myth}
\begin{proof}
	Let $\{I_{\alpha}: \alpha<\mf{c}\}$ be an enumeration of $\mc{B}\cap\mc{I}$ and $\{B_{\alpha}: \alpha<\mf{c}\}$ be an enumeration of $\mc{I}$-positive Borel sets. We shall construct a required partition inductively. At a step $\xi<\continuum$ we choose a sequence $\langle x_{\xi,\zeta,\eta}: \zeta\leq\xi, \eta\leq\xi\rangle$ such that for all $\zeta,\eta\leq\xi$:
$$
x_{\xi,\zeta,\eta}\in B_{\eta}\bez(X_{\xi,\zeta,\eta}\cup\bigcup_{\alpha\leq\xi}I_{\alpha}),
$$
where:
$$
X_{\xi,\zeta,\eta}=\{ x_{\gamma,\beta,\alpha}: \gamma<\xi, \beta\leq\gamma, \alpha\leq\gamma\}\cup\{ x_{\xi,\beta,\alpha}: \beta<\zeta, \alpha\leq\xi\}\cup\{ x_{\xi,\zeta,\alpha}: \alpha<\eta\},
$$
so $X$ is a set of all pointes picked till now. Since:
$$
X_{\xi,\zeta,\eta}\cup\bigcup_{\alpha\leq\xi}I_{\alpha}\in\mc{I}
$$
the choice is always viable. This finishes the cosntruction and for each $\zeta<\continuum$ we set:
$$
L'_{\zeta}=\bigcup_{\xi\geq\zeta}\{x_{\xi,\zeta,\alpha}: \alpha\leq\xi\}.
$$
A family $\{L'_{\alpha}: \alpha<\mf{c}\}$ consists of pairwise disjoint strong $\mc{I}$-Luzin sets. To receive a partition let us enumerate $\RR=\{r_{\alpha}: \alpha<\continuum\}$ and define for each $\xi<\continuum$:
$$
L_{\xi} = \left\{ \begin{array}{ll}
 L'_{\xi} & \textnormal{, if $r_{\xi}\in\bigcup_{\alpha}L'_{\alpha}$.}\\
 L'_{\xi}\cup\{r_{\xi}\} & \textnormal{in the other case.}
 \end{array} \right.
$$
Then $\{L_{\alpha}: \alpha<\continuum\}$ is the partition. 

\end{proof}

\section{Results}

Results will be divided into two groups. First, we will construct subsets of the plane with some additional properties. Second, we will focus on subsets of the line with additional properties involving the algebraic structure of $\RR.$

\subsection{Subsets of the plane}
Let us start with the following observation: 

\begin{myth}[CH]
There exists a set $A\subset\RR^{2}$ such that each horizontal slice $A^{y}$ is a strong $\mc{I}$-Luzin set and each vertical slice $A_{x}$ is a cocountable set. Such a set is $\mc{M}$ and $\mc{N}$-nonmeasurable. Moreover, in the case $\mc{I}=\mc{M}$, $A$ is completely $\mc{M}$-nonmeasurable, and in the case $\mc{I}=\mc{N}$, $A$ is completely $\mc{N}$-nonmeasurable.
\end{myth}

\begin{proof}
Let $L$ be strong $\mc{I}$-Luzin set (obtained by Lemma \ref{superiluzinq}). Let us define:
\begin{eqnarray*}
X & = & \bigcup_{\alpha<\continuum}\{r_{\alpha}\}\times\{r_{\beta}: \beta>\alpha\}
\\
Y & = & L\times \RR.
\end{eqnarray*}
Then $A=X\cup Y$ is the set. Nonmeasurability follows Fubini and Kuratowski-Ulam theorems.

\end{proof}

\begin{myth}[CH]
There exists a set $A\subset\RR^{2}$ such that each vertical slice $A_{x}$ is cocountable and $A$ is completely $\mc{M}$, $\mc{N}$-nonmeasurable.  
\end{myth}
\begin{proof}
Let $\{r_{\alpha}: \alpha<\continuum\}$ be an enumeration of $\RR$, $\{B_{\alpha}: \alpha<\continuum\}$ an enumeration of planar Borel non-meager sets  and $\{L_{\alpha}: \alpha<\continuum\}$ an enumeration of Borel planar sets of positive measure. Let us define:
$$
A'=\bigcup_{\alpha<\continuum}(\{r_{\alpha}\}\times\{r_{\beta}:\beta\geq\alpha\}).
$$
Then let us choose sequences $\langle b_{\alpha}: \alpha<\continuum\rangle$ and $\langle l_{\alpha}: \alpha<\continuum\rangle$ such that for each $\xi<\continuum$:
\begin{eqnarray*}
l_{\xi}&\in&A'\cap L_{\xi}\bez(\bigcup_{\alpha<\xi}(\{l_{\alpha,x}\}\times\RR)),
\\
b_{\xi}&\in&A'\cap B_{\xi}\bez(\bigcup_{\alpha<\xi}(\{b_{\alpha,x}\}\times\RR)),
\end{eqnarray*}
where $l_{\alpha,x}$ and $b_{\alpha,x}$ states for X-coordinate of $l_{\alpha}$ and $b_{\alpha}$. It follows from Fubini and Kuratowski-Ulam theorems that there are uncountably many uncountable vertical slices of $L_{\xi}$ and $B_{\xi}$, so such a choice can be made. This finishes the construction of the sequences and let us denote:
$$
A=A'\bez(\{b_{\alpha}: \alpha<\continuum\}\cup\{l_{\alpha}: \alpha<\continuum\}).
$$
Since in this way we remove at most 2 points from each vertical slice, $A$ is the set.

\end{proof}

\begin{myth}[CH]
There exists a set $A\subset \RR^{2}$ such that each horizontal  slice $A^{y}$ is a strong Luzin set and each vertical slice $A_{x}$ is strong Sierpi\'nski set. Moreover, $A$ is completely $\mc{M}$- and $\mc{N}$-nonmeasurable.
\end{myth}

\begin{proof}
Let $\{L_{\alpha}: \alpha<\continuum\}$ and $\{S_{\alpha}: \alpha<\continuum\}$ be a partition of $\RR$ into strong Luzin sets and strong Sierpi\'nski sets respectively (from Theorem \ref{partycjailuzinych}). Let us set:
$$
A=\bigcup_{\alpha<\continuum}(L_{\alpha}\times S_{\alpha}).
$$
Then $A$ is the set. Complete nonmeasurability follows from Fubini and Kuratowski-Ulam theorems.

\end{proof}

Now, we shall focus on subsets of the plane involving in their constructions Luzin sets. Let us start with some preparation.
\begin{myfact}
	Let $h: X\rightarrow Y$ be a homeomorphism. Then $L$ is a Luzin set in $X$ $\IFF$ $h(L)$ is a Luzin set in $Y$. 
\end{myfact}



Now we may almost effortless prove that:
\begin{myth}
\label{prostesuperluzin}
Assume that a Luzin set exists. Then there exists a set $A\subset\RR^{2}$ such that for each straight line $l$ a set $A\cap l$ is a strong Luzin set.
\end{myth}

\begin{proof}
Let $\{l_{\alpha}:\alpha<\continuum\}$ be an enumeration of all straight lines. Let $L$ be a strong Luzin set (obtained by Lemma \ref{superiluzinq}). Since each straight line is a homeomorphic image of $\RR$ we can take for each $\xi<\continuum$ a set $L_{\xi}$ - image of $L$ by some homeomorphism on its image $h:\RR\rightarrow \RR^{2}$ such that $h(\RR)=l_{\xi}$. Next let us set:
$$
A=\bigcup^{\cdotp}_{\alpha<\continuum}(L_{\alpha}\bez\bigcup_{\beta<\alpha}l_{\beta}).
$$
Each two different lines intersect at most at one point, so for each $\alpha<\continuum$ a set $L_{\alpha}\bez\bigcup_{\beta<\alpha}l_{\beta}$ remains a strong Luzin set in $l_{\alpha}$ and hence $A$ is the set.

\end{proof}

Under CH, we may obtain a little more.
\begin{myth}[CH]
\label{prosteluzinhamel}
 There exists a set $A\subset\RR^{2}$ such that for each straight line $l$ a set $A\cap l$ is a strong Luzin set and $A$ is a Hamel basis.
\end{myth}
Let us state now some other helpful result.
\begin{mylemma}
\label{homeointersekcje}
	Let $f, g: \RR\rightarrow \RR^{2}$ be homeomorphisms on their images and let $M$ be a closed nowhere dense set in $\RR$. Then $g(\RR)\cap f(M)$ is meager in $g(\RR)$. 
\end{mylemma}

\begin{proof}
Let $\langle F_{n}: n\in\NN\rangle$ be a sequence of compact sets such that $\RR=\bigcup_{n\in\NN}F_{n}$. Let us fix some $n\in\NN$. $f(M)$ is closed and nowhere dense in $f(\RR)$ and $f(M)\cap g(F_{n})$ is compact. Suppose a contrario that $f(M)\cap g(F_{n})$ is not nowhere dense in $g(\RR)$, that is it has nonempty interior. Then it also contains some open interval $I$. $I$ is also an interval in $f(\RR)$, what brings us contradiction. Then $f(M)\cap g(F_{n})$ is closed and nowhere dense in $g(\RR)$ and so $f(M)\cap g(\RR)$ is meager in $g(\RR)$.

\end{proof}

Thanks to this lemma we may show that:
\begin{myth}[CH]
\label{homeosuperhamel}
There exists a set $A\subset\RR^{2}$ such that for each homeomorphism $h:\RR\rightarrow \RR^{2}$ on its image a set $h(\RR)\cap A$ is a strong Luzin set and $A$ is a Hamel basis.
\end{myth}

\begin{proof}
Let us enumerate:
\begin{itemize}
\item $\{ h_{\alpha}: \alpha<\continuum\}$ - homeomorphisms from $\RR$ into $\RR^{2}$ on their image,
\item $\{I_{\alpha}: \alpha<\aleph_{0}\}$ - open intervals with rational endpoints,
\item $\{F_{\alpha}: \alpha<\continuum\}$ - closed nowhere dense subsets of $\RR$,
\item  $\{r_{\alpha}: \alpha<\continuum\}$ - $\RR^{2}$.
\end{itemize}
For each $\xi<\continuum$ let us choose a sequence $\langle x_{\xi,\zeta,\eta}: \zeta\leq\xi, \eta<\aleph_{0}\rangle$ and points $a_{\xi}$, $b_{\xi}$ such that for all $\zeta\leq\xi, \eta<\aleph_{0}$:
\begin{enumerate}
\item $x_{\xi,\zeta,\eta}\in h_{\zeta}(I_{\eta})\bez(span(X_{\xi,\zeta,\eta}\cup\{a_{\alpha}, b_{\alpha}: \alpha<\xi\})\cup\bigcup_{\alpha\leq\xi}h_{\alpha}(\bigcup_{\beta\leq\xi}F_{\beta}))$,
\item $a_{\xi}+b_{\xi}=r_{\xi}$,
\item $\{a_{\xi}, b_{\xi}\}\cup Y_{\xi}$ is linearly independent,
\item $a_{\xi}, b_{\xi}\notin\bigcup_{\alpha\leq\xi}h_{\alpha}(\bigcup_{\beta\leq\xi}F_{\beta})$,
\end{enumerate}
where:
\begin{eqnarray*}
X_{\xi,\zeta,\eta}&=&\{ x_{\gamma,\beta,\alpha}: \gamma<\xi, \beta\leq\gamma, \alpha<\aleph_{0}\}\cup\{ x_{\xi,\beta,\alpha}: \beta<\zeta, \alpha<\aleph_{0}\}\cup\{ x_{\xi,\zeta,\alpha}: \alpha<\eta\},
\\
Y_{\xi}&=&\{x_{\gamma,\beta,\alpha}: \gamma\leq\xi, \beta\leq\gamma, \alpha<\aleph_{0}\}\cup\{a_{\alpha}, b_{\alpha}: \alpha<\xi\},
\end{eqnarray*}
that is sets of all points picked till it comes to choose $x_{\xi,\zeta,\eta}$ and $a_{\xi}, b_{\xi}$ respectively.

A set $span(X_{\xi,\zeta,\eta}\cup\{a_{\alpha}, b_{\alpha}: \alpha<\xi\})$ is countable and by Lemma \ref{homeointersekcje} a set $\bigcup_{\alpha\leq\xi}h_{\alpha}(\bigcup_{\beta\leq\xi}F_{\beta})$ is meager in $h_{\zeta}(\RR)$ so $x_{\xi,\zeta,\eta}$ can be picked.

Now let us assume that $r_{\xi}\notin span(Y)$ (in the other case we would consider first point of the plane that doesn't belong to span(Y)). (1) and (2) implies that for all $p,q\in\QQ$ and for all $y\in span(Y)$:
\begin{eqnarray*}
pa_{\xi}+q(r_{\xi}-a_{\xi})+y & \neq & 0,
\\
a_{\xi}(p-q) & \neq & -y-qr_{\xi}.
\end{eqnarray*}
Since $r_{\xi}\notin span(Y)$ we may assume that $p-q\neq 0$. Then:
\begin{eqnarray*}
a_{\xi} & \neq & \frac{-y}{(p-q)}-\frac{q}{(p-q)}r_{\xi},
\end{eqnarray*}
what means that we should pick $a_{\xi}\notin span(Y)-r_{\xi}\QQ$. (3) follows that:
\begin{eqnarray*}
a_{\xi}&\notin&\bigcup_{\alpha\leq\xi}h_{\alpha}(\bigcup_{\beta\leq\xi}F_{\beta}),
\\
a_{\xi}&\notin&r_{\xi}-\bigcup_{\alpha\leq\xi}h_{\alpha}(\bigcup_{\beta\leq\xi}F_{\beta}).
\end{eqnarray*}
and eventually:
\begin{eqnarray*}
a_{\xi}&\notin&(\bigcup_{\alpha\leq\xi}h_{\alpha}(\bigcup_{\beta\leq\xi}F_{\beta})\cup (r_{\xi}-\bigcup_{\alpha\leq\xi}h_{\alpha}(\bigcup_{\beta\leq\xi}F_{\beta}))\cup (span(Y)-r_{\xi}\QQ),
\\
b_{\xi}&=& r_{\xi}-a_{\xi}.
\end{eqnarray*}
This finishes the construction and $A=\bigcup_{\xi<\continuum}(\{x_{\xi,\zeta,\eta}: \zeta\leq\xi, \eta<\aleph_{0}\}\cup\{a_{\xi}, b_{\xi}\})$ is the set.

\end{proof}


%


Now, let us focus on Sierpi\'nski sets. In this case situation is a little different. It is because there are homeomorphisms of the reals which does not preserve measure zero. However, there is a natural class of functions satisfying this condition. 

\begin{myfact}
Let $f: \RR\rightarrow \RR$ be an locally absolutely continuous function and let $\gamma: \RR\rightarrow\RR^{2}$ be such that $\gamma(\RR)=f$. Then:
$$
\mu(A)=\int_{\gamma^{-1}(A)}\sqrt{1+f'^{2}(x)}d\lambda(x)
$$
defines a measure on f. Moreover, $\mu(A)=0$ $\IFF$ $\gamma^{-1}(A)=0$.
\end{myfact}
So now we have a characterization of null sets in given absolutely continuous function. We shall prove the following:
\begin{myth}
Assume that Luzin and Sierpiński sets exist. Then there exist a set $A\subset\RR^{2}$ such that for every increasing and continuous function $f$ $A\cap f$ is a strong Luzin set and for each decreasing locally absolutely continuous function $g$ $A\cap g$ is a strong Sierpiński set. 
\end{myth}
\begin{proof}
Let $\langle f_{\alpha}<\continuum\rangle$ be an enumeration of increasing continuous functions, $\langle g_{\alpha}<\continuum\rangle$ be an enumeration of decreasing locally absolutely continuous functions and $L$ and $S$ be strong Luzin set and strong Sierpi\'nski set respectively. Let us denote for each $\alpha<\continuum$ $F_{\alpha}=f_{\alpha}\upharpoonright {L}$ and $G_{\alpha}=g_{\alpha}\upharpoonright {S}$. Then:
$$
A=\bigcup_{\alpha<\continuum}(F_{\alpha}\bez\bigcup_{\beta\leq\alpha}g_{\beta}\cup G_{\alpha}\bez\bigcup_{\beta\leq\alpha}f_{\beta})
$$
is the set.

\end{proof}

Since existence of Luzin and Sierpi\'nski sets is equivalent to CH, we may achieve stronger assertion. But first we shall state the following lemma:
\begin{mylemma}
\label{wykresyAC}
Let $f, g: \RR\rightarrow \RR$ be locally absolutely continuous functions. Then if $A\subset f$ is a null set in $f$ then $A\cap g$ is a null set in $g$.
\end{mylemma}

Now we may proceed with theorem:
\begin{myth}[CH]
There exist a set $A\subset\RR^{2}$ such that for every increasing continuous function $f$ $A\cap f$ is a strong Luzin set and for each decreasing locally absolutely continuous function $g$ $A\cap g$ is a strong Sierpiński set and $A$ is a Hamel basis.
\end{myth}
\begin{proof}
Let us enumerate:
\begin{itemize}
\item $\{ f_{\alpha}: \alpha<\continuum\}$ - increasing continuous functions,
\item $\{U_{\alpha}: \alpha<\continuum\}$ - open subsets of $\RR$,
\item $\{F_{\alpha}: \alpha<\continuum\}$ - closed nowhere dense subsets of $\RR$,
\item $\{ g_{\alpha}: \alpha<\continuum\}$ - deacreasing locally absolutely continuous functions,
\item $\{B_{\alpha}: \alpha<\continuum\}$ - Borel subsets of $\RR$ of positive measure,
\item $\{N_{\alpha}: \alpha<\continuum\}$ - Borel null subsets of $\RR$,
\item  $\{r_{\alpha}: \alpha<\continuum\}$ - $\RR^2$.
\end{itemize}
For each $\xi<\mf{c}$ let us choose sequences $\langle l_{\xi,\zeta,\eta}: \zeta\leq\xi, \eta\leq\xi\rangle$, $\langle s_{\xi,\zeta,\eta}: \zeta\leq\xi, \eta\leq\xi\rangle$ and points $a_{\xi}$, $b_{\xi}$ such that for all $\zeta, \eta\leq\xi$:
\begin{enumerate}
\item $l_{\xi,\zeta,\eta}\in f_{\zeta}\upharpoonright{U_{\eta}}\bez(span(X_{\xi,\zeta,\eta})\cup\bigcup_{\alpha\leq\xi}f_{\alpha}\upharpoonright{\bigcup_{\beta\leq\xi}F_{\beta}}\cup\bigcup_{\alpha\leq\xi}g_{\alpha})$,
\item $s_{\xi,\zeta,\eta}\in g_{\zeta}\upharpoonright{B_{\eta}}\bez(span(Y_{\xi,\zeta,\eta})\cup\bigcup_{\alpha\leq\xi}g_{\alpha}\upharpoonright{\bigcup_{\beta\leq\xi}N_{\beta}}\cup\bigcup_{\alpha\leq\xi}f_{\alpha})$,
\item $a_{\xi}+b_{\xi}=r_{\xi}$,
\item $\{a_{\xi}, b_{\xi}\}\cup Z_{\xi}$ is linearly independent,
\item $a_{\xi}, b_{\xi}\notin(\bigcup_{\alpha\leq\xi}f_{\alpha}\upharpoonright{\bigcup_{\beta\leq\xi}F_{\beta}}\cup\bigcup_{\alpha\leq\xi}g_{\alpha}\upharpoonright{\bigcup_{\beta\leq\xi}N_{\beta}})$,
\end{enumerate}
where:
\begin{eqnarray*}
X_{\xi,\zeta,\eta}&=&\{ s_{\gamma,\beta,\alpha}: \gamma<\xi, \beta\leq\gamma, \alpha\leq\gamma\}\cup
\{ l_{\gamma,\beta,\alpha}: \gamma<\xi, \beta\leq\gamma, \alpha\leq\gamma\}\cup
\\
& &
\{ l_{\xi,\beta,\alpha}: \beta<\zeta, \alpha\leq\xi\}
\cup\{ l_{\xi,\zeta,\alpha}: \alpha<\eta\}\cup\{a_{\alpha}, b_{\alpha}: \alpha<\xi\},
\\
Y_{\xi,\zeta,\eta}&=&\{ l_{\gamma, \beta, \alpha}: \gamma\leq\xi, \beta\leq\gamma, \alpha\leq\gamma\}\cup
\{ s_{\gamma,\beta,\alpha}: \gamma<\xi, \beta\leq\gamma, \alpha\leq\gamma\}\cup
\\
& &
\{ s_{\xi,\beta,\alpha}: \beta<\zeta, \alpha\leq\xi\}
\cup\{ s_{\xi,\zeta,\alpha}: \alpha<\eta\}\cup\{a_{\alpha}, b_{\alpha}: \alpha<\xi\},
\\
Z_{\xi}&=&\{ l_{\gamma,\beta,\alpha}, s_{\gamma,\beta,\alpha}: \gamma\leq\xi, \beta\leq\gamma, \alpha\leq\gamma\}\cup\{a_{\alpha}, b_{\alpha}: \alpha<\xi\},
\end{eqnarray*}
A reasoning similar to the proof of Theorem \ref{homeosuperhamel} supported by Lemmas \ref{homeointersekcje} and \ref{wykresyAC} shows, that such choices can be made. Then $A=\bigcup_{\xi<\mf{c}}(\{l_{\xi,\beta,\alpha}, s_{\xi,\beta,\alpha}: \beta\leq\xi, \alpha\leq\xi\}\cup\{a_{\xi}, b_{\xi}\})$ is the set.

\end{proof}

We are also able to construct sets similar to these in Theorems \ref{prostesuperluzin} and \ref{prosteluzinhamel} due to fact that each straight line on the plane is an isometry of $\RR$.
\begin{myth}
Assume that a Sierpi\'nski set exists. Then there exists a set $A\subset\RR^{2}$ such that for each straight line $l$ a set $A\cap l$ is a strong Sierpi\'nski set.
\end{myth}
\begin{proof}
	Let $\mathcal{P}=\{ l_{\alpha}: \alpha<\continuum\}$ be an enumeration of all straight lines on the plane and let $\{ T_{\alpha}: \alpha<\continuum\}$ be an enumeration of isometries such that $T_{\alpha}(\RR)=l_{\alpha}$ for $\alpha<\continuum$. By Lemma \ref{superiluzinq} we may assume that we have strong Sierpi\'nski set S. Let us define for each $\xi<\continuum$:
$$
A_{\xi}=T_{\xi}(S)\bez\bigcup_{\alpha<\xi}T_{\alpha}(\RR),
$$
and set:
$$
A=\bigcup_{\alpha<\continuum}A_{\alpha}.
$$
Then $A$ is the set.

\end{proof}

\begin{myth}[CH] 
There exists a set $A\subset\RR^{2}$ such that for each straight line $l$ on the plane a set $l\cap A$ is a strong Sierpi\'nski set and $A$ is a Hamel basis.
\end{myth}
\begin{proof}
Let us enumerate:
\begin{itemize}
\item $\{ l_{\alpha}: \alpha<\continuum\}$ - straight lines on the plane,
\item $\{ T_{\alpha}: \alpha<\continuum\}$ - isometries from $\RR$ into $\RR^{2}$ such that $(\forall \alpha<\continuum)(T_{\alpha}(\RR)=l_{\alpha})$,
\item $\{ N_{\alpha}: \alpha<\continuum\}$ - Borel null subsets of $\RR$,
\item $\{ B_{\alpha}: \alpha<\continuum\}$ - Borel subsets of $\RR$ of positive measure,
\item  $\{r_{\alpha}: \alpha<\continuum\}$ - $\RR^2$.
\end{itemize}
Analogously to previous proofs let us choose for each $\xi<\continuum$ sequence $\langle x_{\xi,\zeta,\eta}: \zeta\leq\xi, \eta\leq\xi \rangle$ and points $a_{\xi}, b_{\xi}$ such that for all $\zeta, \eta\leq\xi$:
\begin{enumerate}
\item $x_{\xi,\zeta,\eta}\in T_{\zeta}(B_{\eta})\bez(span(X_{\xi,\zeta,\eta})\cup\bigcup_{\alpha\leq\xi}T_{\alpha}(\bigcup_{\beta\leq\xi}N_{\beta}))$,
\item $a_{\xi}+b_{\xi}=r_{\xi}$,
\item $\{a_{\xi}, b_{\xi}\}\cup Y_{\xi}$ is linearly independent,
\item $a_{\xi}, b_{\xi}\notin(\bigcup_{\alpha\leq\xi}f_{\alpha}\upharpoonright{\bigcup_{\beta\leq\xi}F_{\beta}}\cup\bigcup_{\alpha\leq\xi}g_{\alpha}\upharpoonright{\bigcup_{\beta\leq\xi}N_{\beta}})$,
\end{enumerate}
where:
\begin{eqnarray*}
X_{\xi,\zeta,\eta}&=&\{ x_{\gamma,\beta,\alpha}: \gamma<\xi, \beta\leq\gamma, \alpha\leq\gamma\}\cup\{ x_{\xi,\beta,\alpha}: \beta<\zeta, \alpha\leq\xi\}\cup
\\
& &\{ x_{\xi,\zeta,\alpha}: \alpha<\eta\}\cup\{ a_{\alpha}, b_{\alpha}: \alpha<\xi\},
\\
Y_{\xi}&=&\{ x_{\gamma,\beta,\alpha}: \gamma\leq\xi, \beta\leq\gamma, \alpha\leq\gamma\}\cup\{ a_{\alpha}, b_{\alpha}: \alpha<\xi\}.
\end{eqnarray*}
Argumentation that at each step such choices can be made is almost identical to these in previous proofs. So, $A=\bigcup_{\xi<\continuum} (\{x_{\xi,\zeta,\eta}: \zeta\leq\xi, \eta\leq\xi\}\cup\{a_{\xi}, b_{\xi}\})$ is the set.

\end{proof}

\subsection{Additional properties involving algebraic structure}
Throughout this section we fix following enumerations:
\begin{itemize}
\item $\{ B_{\alpha}: \alpha<\mf{c}\}$ - Borel $\mc{I}$-positive subsets of $\RR$,
\item $\{ I_{\alpha}: \alpha<\mf{c}\}$ - Borel subsets of $\RR$ that belong to $\mc{I}$,
\item $\{ r_{\alpha}: \alpha<\mf{c}\}$ - $\RR$.
\end{itemize}
Let us begin with two helpful lemmas.
\begin{mylemma}
	Let $L$ be an $\mc{I}$-Luzin set. Then there exists a linearly independent $\mc{I}$-Luzin set.
\end{mylemma}

\begin{mylemma}
	Let $L$ be an $\mc{I}$-Luzin set. Then there exists a linearly independent strong $\mc{I}$-Luzin set.
\end{mylemma}

\begin{proof}
	By Lemma \ref{superiluzinq} we may assume that we have a strong $\mc{I}$-Luzin set $L_0$. For each $\xi<\continuum$ we choose a sequence $\langle l_{\xi,\zeta}: \zeta\leq\xi\rangle$ such that for every $\zeta\leq\xi$:
$$
l_{\xi,\zeta}\in (L_0\cap B_{\zeta})\bez span(\{l_{\beta,\alpha}: \beta<\xi, \alpha\leq\beta\}\cup\{l_{\xi,\alpha}: \alpha<\zeta\}).
$$
Then $L=\{l_{\xi,\zeta}: \xi<\mf{c}, \zeta\leq\xi\}$ is the set.

\end{proof}

In this context the following question can be posed.
\begin{problem}
Does the existence of an $\mc{I}$-Luzin set imply the existence of an $\mc{I}$-Luzin set which is a Hamel base?
\end{problem}

Next observation is known and can be found e.g. in \cite{Kucz}.

\begin{myfact}[CH]
There is an $\mc{I}$-Luzin set $L$ such that $L$ is a linear subspace of $\RR.$ 
\end{myfact}

Let us notice that the existance of a Luzin set which is a linear subspace of $\RR$ does not imply $CH.$

\begin{myth}
 It is consistent that $2^\omega=\omega_2$ and there is  a Luzin set which is a linear subspace of $\RR.$
\end{myth}
\begin{proof}
 Let us work in a model $V'$ obtained from a model $V$ of $CH$ by adding $\omega_2$ Cohen reals $\{ c_\alpha:\ \alpha<\omega_2\}$. Set 
 $$
 L=span_{\mathbb Q}(\{c_\alpha:\ \alpha<\omega_2\}).
 $$ 
 We claim that $L$ is a Luzin set. Indeed, take any meager Borel set $A$. Its Borel code is in $V[\{c_\alpha:\ \alpha\in \tilde{A}\}],$ where $\tilde{A}$ is a countable subset of $\omega_2.$ We claim that 
$$L\cap A\subseteq span_{\mathbb{Q}}(\{c_\alpha:\ \alpha\in\tilde{A}\}).
$$ 
Let $x\in L\setminus span_{\mathbb{Q}}(\{c_\alpha:\ \alpha\in\tilde{A}\}).$ Then $x=y+qc_\beta$, where $q\in\mathbb{Q}\setminus\{0\}$ and $\beta\notin\tilde{A}$ and $y\in Span_{\mathbb Q}(\{c_\alpha:\ \alpha\neq\beta\}.$ Notice that $c_\beta\notin\frac{1}{q}(A-y),$ what implies that $x\notin A.$
\end{proof}

The above observation leads us to a natural question.
\begin{problem}
 Does the existance of a Luzin set imply the existance of a Luzin set which is a linear subspace of R?
\end{problem}

The following theorem shows how partition $\RR$ into translations of an $\mc{I}$-Luzin set.
\begin{myth}
Let $L$ be a linearly independent $\mc{I}$-Luzin set. Then there exists a set $X$ such that $\{x+L: x\in X\}$ is a partition of $\RR$.
\end{myth}

\begin{proof}
We shall construct inductively a sequence $\langle x_{\alpha}: \alpha<\continuum\rangle$. Let $x_{0}=r_{0}-l$ for some $l\in L$. At the step $\xi<\continuum$ we want to choose such a $x_{\xi}$ that for each $\alpha<\xi$ we gonna have:
$$
(x_{\xi}+L)\cap(x_{\alpha}+L)=\emptyset
$$
and $x_{\xi}+l=r_{\xi}$ for some $l\in L$. Let us fix some $l\in L$, set $x_{\xi}=r_{\xi}-l$ and suppose that for some $\alpha<\xi$:
$$
(x_{\xi}+L)\cap(x_{\alpha}+L)\neq\emptyset.
$$
It means that there exist some $l', l''\in L$ that:
\begin{eqnarray*}
r_{\xi}-l+l' & = & x_{\alpha}+l'',
\\
r_{\xi}-x_{\alpha} & = & l-l'+l''.
\end{eqnarray*}
Such a representation is unique. Let us denote for each $\alpha<\xi$ a set $Y_{\alpha}$ of these elements that take part in the representation of $r_{\xi}-x_{\alpha}$. Then it is possible and sufficient to choose $l\in L\bez\bigcup_{\alpha<\xi}Y_{\alpha}$ and set $x_{\xi}=r_{\xi}-l$.
\\
This finishes the construction and $X=\{x_{\alpha}: \alpha<\continuum\}$ is the set.

\end{proof}

\begin{myremark}
We could use linearly independent strong $\mc{I}$-Luzin set as well.
\end{myremark}

Assuming CH implies that $X$ also may be a $\mc{I}$-Luzin set.
\begin{myth}[CH] For each $\mc{I}$-Luzin set $L$ there exists an $\mc{I}$-Luzin set $X$ such that $\{x+L: x\in X\}$ is a partition of $\RR$.
\end{myth}

\begin{proof}
We will construct inductively a sequence $\langle x_{\alpha}: \alpha<\continuum \rangle$. For each $\xi<\continuum$ we choose $x_{\xi}$ such that:
\begin{enumerate}
\item $x_{\xi}\notin \bigcup_{\alpha<\xi}M_{\alpha}$,
\item $(\forall \alpha<\xi)((x_{\xi}+L)\cap(x_{\alpha}+L)=\emptyset)$,
\item $x_{\xi}+l=r_{\xi}$ for some $l\in L$.
\end{enumerate}
To meet these conditions let us consider some helpful sets. Similarly to previous proof we define for each $\alpha<\xi$ a set $Y_{\alpha}$ of these elements that take part in a representation of $r_{\xi}-x_{\alpha}$. Moreover, denote $Z=\{l\in L: r_{\xi}-l\in\bigcup_{\alpha<\xi}M_{\alpha}\}$. Then we can set $x_{\xi}=r_{\xi}-l$ for some $l\in L\bez(\bigcup_{\alpha<\xi}Y_{\alpha}\cup Z)$.

This finishes the construction and $X=\{x_{\alpha}: \alpha<\continuum\}$ is the set.
\end{proof}

\begin{myth}[CH] 
There exists an $\mc{I}$-Luzin set $L$ such that $L+L$ is an $\mc{I}$-Luzin set.
\end{myth}
\begin{proof}
Let us choose a sequence $\langle l_{\alpha}: \alpha<\continuum\rangle$ such that for each $\xi<\continuum$:
$$
	l_{\xi}\notin(\bigcup_{\alpha<\xi}I_{\alpha}-(\{l_{\alpha}: \alpha<\xi\}\cup\{0\}))\cup\frac{1}{2}\bigcup_{\alpha<\xi}I_{\alpha}.
$$
Then $L=\{l_{\alpha}: \alpha<\continuum\}$ is the set.

\end{proof}

The following theorem is actually a folklore, but in the name of completeness we give its (short) proof.
\begin{myth}[CH] 
There exists an $\mc{I}$-Luzin set $L$ such that $L+L=\RR$.
\end{myth}

\begin{proof}
Let us choose two sequences $\langle l_{\alpha}: \alpha<\continuum\rangle$ and $\langle l'_{\alpha}: \alpha<\continuum\rangle$ such that for each $\xi<\continuum$:
\begin{enumerate}
\item $l_{\xi}, l'_{\xi}\notin\bigcup_{\alpha<\xi}I_{\alpha}$,
\item $l_{\xi}+l'_{\xi}=r_{\xi}$.
\end{enumerate}
Condition (1) implies that $l'_{\xi}=r_{\xi}-l_{\xi}$ and by (2) we should pick:
$$
l_{\xi}\notin\bigcup_{\alpha<\xi}I_{\alpha}\cup(r_{\xi}-\bigcup_{\alpha<\xi}I_{\alpha}).
$$
This finishes the construction and $L=\{l_{\alpha}: \alpha<\continuum\}\cup\{l'_{\alpha}: \alpha<\continuum\}$ is the set.

\end{proof}

\begin{mylemma}
For each $A$, $B$:
$$
\bigoplus^{n}(A\cup B)=\bigcup_{k=0}^{n}(\bigoplus^{k}A+\bigoplus^{n-k}B).
$$
\end{mylemma}

\begin{myth}[CH] 
For each $n\in\NN\bez\{0\}$ There exists an $\mc{I}$-Luzin set $L$  such that $\bigoplus^{n}L$ is an $\mc{I}$-Luzin set and $\bigoplus^{n+1}L=\RR$.
\end{myth}

\begin{proof}
Let us choose for each $\xi<\continuum$ a sequence $\langle l_{k}^{\xi}: 1\leq k\leq n+1 \rangle$ such that for all $i$, $0\leq i \leq n-1$ and for all $J$, $1\leq j \leq n+1$:
\begin{itemize}
\item $(n-i)l_{j}^{\xi}+\bigoplus^{i}(L_{\xi}\cup\{l_{k}^{\xi}: 1\leq k\leq j-1\})\subset(\bigcup_{\alpha<\xi}I_{\alpha})^{c}$,
\item $\Sigma_{k=1}^{n+1}l_{k}^{\xi}=r_{\xi}$,
\end{itemize}
where $L_{\xi}=\{l_{k}^{\alpha}: 1\leq k \leq n+1 \land \alpha<\xi\}$. Let us check what are the conditions for $l_{n+1}^{\xi}$. For each $i$, $0\leq i \leq n-1$:
$$
(n-i)(r_{\xi}-\sum_{k=1}^{n}l_{k}^{\xi})+\bigoplus^{i}(L_{\xi}\cup\{l^{\xi}_{1}, \ldots, l^{\xi}_{n}\})\subset(\bigcup_{\alpha<\xi}I_{\alpha})^{c}.
$$
Since:
\begin{eqnarray*}
\bigoplus^{i}(L_{\xi}\cup\{l^{\xi}_{1}, \ldots, l^{\xi}_{n}\}) & = & \bigcup_{j=0}^{i}(\bigoplus^{j}(L_{\xi}\cup\{l^{\xi}_{1}, \ldots,l^{\xi}_{n-1}\})+\bigoplus^{i-j}\{l^{\xi}_{n}\})
\\
& = & \bigcup_{j=0}^{i}(\bigoplus^{j}(L_{\xi}\cup\{l^{\xi}_{1}, \ldots,l^{\xi}_{n-1}\})+(i-j)l^{\xi}_{n})
\end{eqnarray*}
we have:
\begin{eqnarray*}
(n-i)(r_{\xi}-\sum_{k=1}^{n}l_{k}^{\xi})+\bigoplus^{i}(L_{\xi}\cup\{l^{\xi}_{1}, \ldots, l^{\xi}_{n}\}) & =
\\ =(n-i)(r_{\xi}-\sum_{k=1}^{n}l_{k}^{\xi})+\bigcup_{j=0}^{i}(\bigoplus^{j}(L_{\xi}\cup\{l^{\xi}_{1}, \ldots,l^{\xi}_{n-1}\})+(i-j)l^{\xi}_{n})= & 
\\
=\bigcup_{j=0}^{i}((n-i)(r_{\xi}-\sum_{k=1}^{n}l_{k}^{\xi})+\bigoplus^{j}(L_{\xi}\cup\{l^{\xi}_{1}, \ldots,l^{\xi}_{n-1}\})+(i-j)l^{\xi}_{n}).
\end{eqnarray*}
The latest is a subset of $(\bigcup_{\alpha<\xi}I_{\alpha})^{c}$ if for all $j$, $ 0\leq j\leq i$:
\begin{eqnarray*}
(n-i)(r_{\xi}-\sum_{k=1}^{n}l_{k}^{\xi})+\bigoplus^{j}(L_{\xi}\cup\{l^{\xi}_{1}, \ldots,l^{\xi}_{n-1}\})+(i-j)l^{\xi}_{n} & \subset & (\bigcup_{\alpha<\xi}I_{\alpha})^{c},
\\
(n-i)(r_{\xi}-\sum_{k=1}^{n}l_{k}^{\xi})+(i-j)l^{\xi}_{n} & \notin & \bigcup_{\alpha<\xi}I_{\alpha}-\bigoplus^{j}(L_{\xi}\cup\{l^{\xi}_{1}, \ldots,l^{\xi}_{n-1}\})
\end{eqnarray*}
and finally:
$$
(n-j)l_{n}^{\xi}\notin\bigcup_{\alpha<\xi}M_{\alpha}-\bigoplus^{j}(L_{\xi}\cup\{l^{\xi}_{1}, \ldots,l^{\xi}_{n-1}\})+(i-n)r_{\xi}+(n-i)\sum_{k=1}^{n-1}l_{k}^{\xi},
$$
so if we are able to choose rest of the points, then we may also pick proper $l_{n}^{\xi}$. Let us check then the conditions for $l_{k}^{\xi}$, $k<n$. We have that for all $j$, $ 0\leq j\leq n-1$:
\begin{eqnarray*}
(n-j)l_{k}^{\xi}+\bigoplus^{j}(L_{\xi}\cup\{l_{1}^{\xi}, \ldots, l_{k}^{\xi}\})&\subset&(\bigcup_{\alpha<\xi}I_{\alpha})^{c},
\\
(n-j)l_{k}^{\xi}+\bigcup_{i=0}^{j}(\bigoplus^{i}(L_{\xi}\cup\{l_{1}^{\xi}, \ldots, l_{k-1}^{\xi}\})+\bigoplus^{j-i}l_{k}^{\xi})&\subset&(\bigcup_{\alpha<\xi}I_{\alpha})^{c},
\\
\bigcup_{i=0}^{j}(\bigoplus^{i}(L_{\xi}\cup\{l_{1}^{\xi}, \ldots, l_{k-1}^{\xi}\})+\bigoplus^{j-i}l_{k}^{\xi})&\subset&(\bigcup_{\alpha<\xi}I_{\alpha})^{c},
\\
\bigcup_{i=0}^{j}(\bigoplus^{i}(L_{\xi}\cup\{l_{1}^{\xi}, \ldots, l_{k-1}^{\xi}\})+(n-i)l_{k}^{\xi})&\subset&(\bigcup_{\alpha<\xi}I_{\alpha})^{c}.
\end{eqnarray*}
and the latest is true if for all $i$, $ 0\leq i\leq j$:
$$
(n-i)l_{k}^{\xi}\notin\bigcup_{\alpha<\xi}M_{\alpha}-\bigoplus^{i}(L_{\xi}\cup\{l_{1}^{\xi}, \ldots, l_{k-1}^{\xi}\}).
$$
It means that we are also able to choose each $l_{k}^{\xi}$ if only we may choose points of subindexes lower than $k$. And since our only condition for $l_{1}^{\xi}$ is that for all $i$, $0\leq i\leq n-1$:
$$
(n-i)l_{1}^{\xi}\notin\bigcup_{\alpha<\xi}I_{\alpha}-L_{\xi}
$$
we may choose the whole sequence.
\\
This finishes the construction and $L=\bigcup_{\alpha<\continuum}\{l_{1}^{\alpha}, l_{2}^{\alpha}, \ldots, l_{n+1}^{\xi}\}$ is the set.

\end{proof}

\begin{myth}[CH]
\label{luzinliniowap}
There exists an $\mc{I}$-Luzin set $L$ such that $span(L)$ is an $\mc{I}$-Luzin set.
\end{myth}

\begin{proof}
We shall choose a sequnce $\langle l_{\alpha}: \alpha<\continuum\rangle$ such that for each $\xi<\continuum$:
\begin{enumerate}
\item $\{l_{\xi}\cup\{l_{\alpha}: \alpha<\xi\}$ is linearly independent,
\item $(\forall q\in\QQ, q\neq 0)(\forall l\in span(\{l_{\alpha}: \alpha<\xi\})(ql_{\xi}+l\notin\bigcup_{\alpha<\xi}I_{\alpha})$
\end{enumerate}
(1) implies that $l_{\xi}\notin span(\{l_{\alpha}: \alpha<\xi\})$, (2) follows that $l_{\xi}\notin\QQ\cdot\bigcup_{\alpha<\xi}I_{\alpha}+span(\{l_{\alpha}: \alpha<\xi\})$. Both sets are from $\mc{I}$, so such choice can be made.
\\
This finishes the construction and $L=\{l_{\alpha}: \alpha<\continuum\}$ is the set.

\end{proof}

\begin{mycoro}[CH] 
\begin{enumerate}
\item There exists an $\mc{I}$-Luzin set $L$ such that $\bigoplus^{n+1}L$ is an $\mc{I}$-Luzin for each $n\in\NN$,
\item There exists an $\mc{I}$-Luzin set $L$ such that $L+L=L$,
\item There exists an $\mc{I}$-Luzin set $L$ such that $\langle\bigoplus^{n+1}L: n\in\NN\rangle$ is a ascending sequence of $\mc{I}$-Luzin sets.
\end{enumerate}
\end{mycoro}
\begin{proof}
Let $L$ be a $\mc{I}$-Luzin set as in Theorem \ref{luzinliniowap}.
\begin{enumerate}
\item $L$ satisfies that.
\item $span(L)$ has such a property.
\item It suffices to take $L\cup\{0\}$.
\end{enumerate}
\end{proof}

\begin{myth}[CH]
There exists a Luzin set $L$ such that $L+L$ is a Bernstein set.
\end{myth}

\begin{proof}
Let $\{P_{\alpha}: \alpha<\continuum\}$ be an enumeration of perfect sets. We choose sequences $\{l_{\alpha}: \alpha<\continuum\}$, $\{l'_{\alpha}: \alpha<\continuum\}$ and $\{p_{\alpha}: \alpha<\continuum\}$ such that for each $\xi<\continuum$:
\begin{enumerate}
\item $l_{\xi}, l'_{\xi}\notin\bigcup_{\alpha<\xi}M_{\alpha}$,
\item $(\bigcup_{\alpha\leq\xi}\{l_{\alpha}, l'_{\alpha}\}+\bigcup_{\alpha\leq\xi}\{l_{\alpha}, l'_{\alpha}\})\cap\{p_{\alpha}: \alpha<\xi\}=\emptyset$,
\item $l_{\xi}+l'_{\xi}\in P_{\xi}$,
\item $p_{\xi}\in P_{\xi}$.
\end{enumerate}
Suppose that we are at the step $\xi<\continuum$. Let us observe that for fixed:
$$
l'_{\xi}\notin\bigcup_{\alpha<\xi}M_{\alpha}\cup(\{p_{\alpha}: \alpha<\xi\}-(\{l_{\alpha}:\alpha<\xi\}\cup\{l'_{\alpha}:\alpha<\xi\}))\cup\frac{1}{2}\{p_{\alpha}:\alpha<\xi\}
$$
it is sufficient to choose:
$$
l_{\xi} \in (\bigcup_{\alpha<\xi}M_{\alpha})^{c}\cap(P_{\xi}-l'_{\xi})\bez
((\{p_{\alpha}: \alpha<\xi\}-(\{l_{\alpha}: \alpha<\xi\}\cup\{l'_{\alpha}: \alpha\leq\xi\}))\cup\frac{1}{2}\{p_{\alpha}:\alpha<\xi\}).
$$
Let us denote:
\begin{eqnarray*}
M_{1}& = &\bigcup_{\alpha<\xi}M_{\alpha},
\\
M_{2}& = &\bigcup_{\alpha<\xi}M_{\alpha}\cup(\{p_{\alpha}: \alpha<\xi\}-(\{l_{\alpha}:\alpha<\xi\}\cup\{l'_{\alpha}:\alpha<\xi\}))\cup\frac{1}{2}\{p_{\alpha}:\alpha<\xi\},
\\
P& = &P_{\xi},
\\
C& = & ((\{p_{\alpha}: \alpha<\xi\}-(\{l_{\alpha}: \alpha<\xi\}\cup\{l'_{\alpha}: \alpha\leq\xi\}))\cup\frac{1}{2}\{p_{\alpha}:\alpha<\xi\}).
\end{eqnarray*}
It is clear now that we may reduce our task to the following: does there exist  $l'\in M_{2}^{c}$ such that a set $M_{1}^{c}\cap(P-l')$ has cardinality $\continuum$?

To prove it let us call our universe $V$. We can extend it via generic extension to $V'$ such that $V'\models cov(\mc{M})\ge\omega_2.$ We will work in $V'$. $M_1$ and $M_2$ are now Borel meager sets decoded in $V'$ and $P$ is a perfect set decoded in $V'$. (Their codes are from $V$.) Let us now fix a set $A\subseteq P$ of cardinality $\omega_1.$ Notice that
for every $a\in A$ a set $\{l:\; a-l\in M_1^c\}=-M_1^c+a$ is comeager. Since $cov(\mc{M})>\omega_1$ we get that 
a set 
$$\bigcap_{a\in A}\{l:\; a-l\in M_1^c\}\cap M_2^c\neq\emptyset.$$
It shows that $V'\models \exists l'\in M_2^c |M_1^c\cap (P-l')|\ge\omega_1.$ A set $M_1^c\cap (P-l')$ is uncountable Borel set, so it contains a perfect set. It gives us an  observation that $V'$ models the following sentence:
$$
(\exists l'\in\RR)(\exists T\in Perf)(\forall x\in \RR )(l'\in M_2^c\land (x\in T\rightarrow x\in M_1^c\;\land\; x+l'\in P))
$$
The latter sentence is a $\Sigma^1_2$-sentence, so by Shoenfield absoluteness theorem it is also true in $V,$ what completes the proof.
  
\end{proof}

\begin{myth}[CH]
There exists a Luzin set $S$ such that $S+S$ is a Bernstein set.
\end{myth}

\begin{proof}
It goes almost identically to the previous one. In crucial moment we have to find out whether there exists $s'\in N_{2}^{c}$, $N_{2}^{c}$- conull set, such that $N_{1}^{c}\cap(P-s')$ has a cardinality $\continuum$, where $N_{1}^{c}$ is conull and $P$ is perfect.

It can be proven using assumption that $cov(\mc{N})\ge\omega_2$ and Shoenfield absoluteness theorem.

\end{proof}

Simillar reasoning may be found in \cite{KRSZ}.




\begin{myfact}
\label{invariant+}
Let $A$ and $B$ be sets and $X=\bigcup_{n\in\NN}(A+\bigoplus^{n}B)$. Then $X+B\subset X$.
\end{myfact}

\begin{mylemma}
\label{invariantcomeager}
There exists a comeager null set $R$ and perfect nowhere dense null set $P$ such that $R+P\subset R$.
\end{mylemma}

\begin{proof}
Let $\langle q_{n}: n\in\NN\rangle$ be an enumeration of $\QQ$. We shall construct a set $A=\bigcap_{n\in\NN}G_{n}$, where $G_{n}=\bigcup_{k\in\NN}B(q_{k}, r_{n,k})$, and a Cantor type set $C=\bigcap_{n\in\NN}C_{n}$, where $C_{0}=[0, 1]$ and for $n>0$ a set $C_{n}$ is made through erasing an open interval from the middle of each closed interval that $C_{n-1}$ is consisted of, so every $C_{n}$ is a sum of $2^{n}$ closed intervals and $\langle C_{n}: n\in\NN\rangle$ is a sequence of compact nested sets.
\\
Let us denote the length of a closed interval in $C_{n}$ by $d_{n}$. We want to set the values of $r_{n,k}$ and $d_{n}$ so that for each $j\in\NN$ set $A+\bigoplus^{j}C$ is null. Since $0\in C$ it is sufficient to only take care of $A+\bigoplus^{2^{j}}C$ for every $j\in\NN$. Let us observe that:
$$
A+\bigoplus^{2^{j}}C\subset\bigcap_{n\in\NN}\bigcup_{k\in\NN}(B(q_{k},r_{n,k})+\bigoplus^{2^{j}}C).
$$
Let us fix some $i\in\NN$. A set $C_{i}$ is a sum of $m=2^{i}$ closed intervals of length $d_{i}$ each. Since for every interval $I$ set $I+I$ is an interval with double length, $C_{i}+C_{i}$ is a sum of $\frac{(m+1)m}{2}$ closed intervals, each of length $2d_{i}$. It follows that we can give upper bound for the number of intervals that $\bigoplus^{2^{j}}C_{i}$ is consisted of and it is $m^{2^{j}}=2^{2^{j}i}$, each of length $2^{j}d_{i}$. Let us note also that if we take an interval $(a, b)$ of length $\geq 2^{j-1}d_{i}$ then:
$$
\lambda((a, b)+\bigoplus^{2^{j}}C)\leq\lambda((a, b)+\bigoplus^{2^{j}}C_{i}).
$$
Now let us set for each $n, k\in\NN$:
\begin{eqnarray*}
d_{n}&=&\frac{1}{2^{2^{n}n+2n+1}},
\\
r_{n,k}&=&2^{n-2}d_{n+k}.
\end{eqnarray*}
Then for fixed $j\in\NN$ we have:
\begin{eqnarray*}
\lambda(A+\bigoplus^{2^{j}}C) & \leq & \lambda(\bigcap_{n\in\NN}\bigcup_{k\in\NN}(B(q_{k},r_{n,k})+\bigoplus^{2^{j}}C))
\\
 & = & \lambda(\bigcap_{n\geq j}\bigcup_{k\in\NN}(B(q_{k},r_{n,k})+\bigoplus^{2^{j}}C)),
\end{eqnarray*}
so for fixed $n=j+i$ we have:
\begin{eqnarray*}
\lambda(\bigcup_{k\in\NN}(B(q_{k},r_{j+i,k})+\bigoplus^{2^{j}}C)) & \leq & \lambda(\bigcup_{k\in\NN}(B(q_{k},r_{j+i,k})+\bigoplus^{2^{j+i}}C_{j+i+k}))
\\
 & \leq & \sum_{k=0}^{\infty}(4r_{j+i,k}2^{2^{j+i}(j+i+k)}+2^{2^{j+i}(j+i+k)}2^{j+i}d_{j+i+k})
\\
& = & \sum_{k=0}^{\infty}2^{2^{j+i}(j+i+k)}2^{j+i+1}d_{j+i+k}
\\
& = & \sum_{k=0}^{\infty}\frac{2^{2^{j+i}(j+i+k)}2^{j+i+1}}{2^{2^{j+i+k}(j+i+k)}2^{2j+2i+2k+1}}
\\
& \leq & \frac{1}{2^{j+i}}\sum_{k=0}^{\infty}\frac{1}{2^{k}},
\end{eqnarray*}
what tends to $0$ if $i\rightarrow\infty$. This proofs that for each $n\in\NN$ a set $A+\bigoplus^{n}C$ is null and since $A$ is comeager we may put by Lemma \ref{invariant+} $R=\bigcup_{n\in\NN}(A+\bigoplus^{n}C)$ and $P=C$ which are the desired sets.
\end{proof}

A little stronger version of the previous lemma can be proven. However, the proof uses absoluteness theorem.
\begin{mylemma}
 Let $A$ be a null set. We can find a perfect set $P$ such that for every $n$
 $$ A+\underbrace{P+P+\cdots +P}_n\in \mc{N}.
 $$
\end{mylemma}
\begin{proof}
 We can assume that $A$ is Borel. Let $V$ be our universe. We enlarge it (via forcing) to $V'$ satisfying $V'\models add(\mc{N})=\omega_3.$ 
 
 Let us work now in $V'.$ Take $X\subseteq \RR$ of cardinality $\omega_2.$ Then $A+X\in \mc{N},$ so we can find a null Borel set $B,$ such that $A+X\subseteq B.$ Notice that $\{x:\ x+A\subseteq B\}$ is a coanalytic set of cardinality $\omega_2,$ hence, it contains a perfect set $P_0.$
 
 Now, set $A_1=A_0+P_0.$ We want to find a perfect set $P_1\subseteq P_0$ such that $A_1+P_1\in \mc{N}.$ Moreover, we require that the first splitting node in $P_0$ is still a splitting node in $P_1$.
 We procced by a simple induction on n-th step finding for a null set $A_n$ and a perfect set $P_n$ a perfect set $P_{n+1}\subseteq P_n$ such that $A_{n+1}=A_n+P_{n+1}$ is null and all splitting nodes from first $n+1$ levels in $P_n$ remains splitting nodes in $P_{n+1}$.
 
 Let us notice that this construction gives us a sequence of perfect sets $(P_n, n\in\omega)$ such that $P=\bigcap_{n\in\omega}P_n$ is a perfect set. Moreover, we can find a null $G_\delta$ $B$ such that $B\supseteq\bigcup_{n\in\omega}A_n.$
Notice that
 $$
 V'\models (\exists P\in Perf)(\exists B)(\forall n)(\forall x)(\forall a)(\forall b)
 (B\text{ is null $G_\delta$} \land 
 $$
 $$
 \quad (a\in A\land b\notin B\land x_0,x_1,\ldots,x_n\in P\rightarrow a+x_0+\cdots x_n\neq b)),
 $$
 where $x_0, x_1,\ldots, x_n$ are naturally coded by $x$ e.g. by the formula $x_i(k)=x(kn+i).$
 
 Above formula is $\Sigma^1_2$, so by Shoenfield absoluteness theorem it holds also in $V,$ what finishes the proof. 
\end{proof}

In a similar way we can prove a lemma for meager sets.
\begin{mylemma}
 Let $A$ be a meager set. We can find a perfect set $P$ such that for every $n$
 $$ A+\underbrace{P+P+\cdots +P}_n\in \mc{M}.
 $$
\end{mylemma}

The following theorem follows from a stronger assertion that for L - Luzin set and S - Sierpi\'nski set a set $L\times S$ is Menger (see \cite{Babin}), and so is $L+S$ and Menger sets are not Bernstein. We give a simple proof of this statement.
\begin{myth}
There are no Luzin set $L$ and Sierpi\'nski set $S$ such that $L+S$ is a Bernstein set.
\end{myth}

\begin{proof}
Let $L$ and $S$ be Luzin and Sierpi\'nski sets respectively and let $R$ and $P$ be sets as in Lemma \ref{invariantcomeager}. Let us denote $A=-R$ and $B=-A^{c}$. Then $A-P\subset A$ and so for any $p\in P$ we have $(A-p)\cap(-B)=\emptyset$. It follows that for every $a\in A$ and $b\in B$ $a-p\neq -b$, hence $a+b\neq p$ and therefore $P\subset (A+B)^{c}$. We will show that $L+S$ also contains some perfect set.
\begin{eqnarray*}
L+S&=&((L\cap A)\cup(L\cap A^{c}))+((S\cap B)\cup(S\cap B^{c}))
\\
& = &((L\cap A)+(S\cap B))\cup((L\cap A)+(S\cap B^{c}))\cup
\\
&&\cup((L\cap A^{c})+(S\cap B))\cup((L\cap A^{c})+(S\cap B^{c}))
\end{eqnarray*}
$(L\cap A)+(S\cap B)\subset A+B$ and sets $(L\cap A)+(S\cap B^{c})$, $(L\cap A^{c})+(S\cap B)$ and $(L\cap A^{c})+(S\cap B^{c})$ are Luzin, Sierpi\'nski and countable respectively, so they have countable intersection with $P$. It follows that there exists a perfect set $P'\subset P$ such that $(L+S)\cap P'=\emptyset$ and therefore $L+S$ cannot be a Bernstein set.

\end{proof}

The last theorem clearly implies that:
\begin{mycoro}
\label{l+snieprosta}
There are no Luzin set $L$ and Sierpi\'nski set $S$ such that $L+S=\RR$.
\end{mycoro}
This result also follows from the fact that for every Luzin set $L$ and a meager set $M$, $L+M\neq\RR$ (see \cite{Gal}) and for every Sierpi\'nski set $S$ and a null set $N$, $L+N\neq\RR$ (see \cite{Paw}).

\end{document}